\documentclass[11pt]{amsart}
\usepackage{amsmath, amsthm, amssymb, amsfonts}
\usepackage[normalem]{ulem}
\usepackage{hyperref}
\usepackage{verbatim}

\usepackage{mathtools}

\usepackage{tikz-cd}
\usetikzlibrary{arrows}
\usepackage{lmodern}
\usetikzlibrary{decorations.pathmorphing}
\tikzset{snake it/.style={decorate, decoration=snake}}

\theoremstyle{plain}
\newtheorem{thm}{Theorem}
\newtheorem{cor}[thm]{Corollary}
\newtheorem{lemma}[thm]{Lemma}

\newtheorem{prop}[thm]{Proposition}
\newtheorem{question}[thm]{Question}
\theoremstyle{definition}

\theoremstyle{remark}

\newcommand{\BQ}{{\mathbb{Q}}}

\newcommand{\BZ}{{\mathbb{Z}}}

\newcommand{\CE}{{\mathcal E}}

\newcommand{\CL}{{\mathcal L}}
\newcommand{\CM}{{\mathcal M}}

\newcommand{\CO}{{\mathcal O}}

\DeclareFontFamily{OT1}{rsfs}{}
\DeclareFontShape{OT1}{rsfs}{n}{it}{<-> rsfs10}{}
\DeclareMathAlphabet{\curly}{OT1}{rsfs}{n}{it}

\newcommand{\p}{\mathbb{P}}

\newcommand\Spec{\operatorname{Spec}}

\newcommand{\Pic}{\mathop{\rm Pic}\nolimits}

\newcommand{\Aut}{\mathop{\mathrm{Aut}}}

\begin{document}
\baselineskip=14.6pt

\title[Automorphisms of Hilbert schemes of points on surfaces]
{Automorphisms of Hilbert schemes of points on surfaces}

\author[P.~Belmans]{Pieter Belmans}
\address{Mathematisches Institut, Universit\"at Bonn, Endenicher Allee 60, 53115 Bonn, Germany}
\email{pbelmans@math.uni-bonn.de}

\author[G.~Oberdieck]{Georg Oberdieck}
\address{Mathematisches Institut, Universit\"at Bonn, Endenicher Allee 60, 53115 Bonn, Germany}
\email{georgo@math.uni-bonn.de}
\author[J.~V.~Rennemo]{J\o rgen Vold Rennemo}
\address{Department of Mathematics, University of Oslo, PO Box 1053 Blindern, 0316 Oslo, Norway}
\email{jorgeren@uio.no}

\date{\today}
\begin{abstract}
We show that every automorphism of the Hilbert scheme of $n$ points on a weak Fano or general type surface is natural, i.e.~induced by an automorphism of the surface,
unless the surface is a product of curves and $n=2$.
In the exceptional case there exists a unique non-natural automorphism.
More generally,
we prove that
any isomorphism between Hilbert schemes of points on smooth projective surfaces,
where one of the surfaces is weak Fano or of general type and not equal to the product of curves,
is natural.
We also show that
every automorphism of the Hilbert scheme of $2$ points on $\p^n$ is natural.
\end{abstract}

\maketitle

\setcounter{tocdepth}{1}
\tableofcontents

\section{Introduction}
Let $X$ be a non-singular complex projective surface and let $X^{[n]}$ be the Hilbert scheme of $n$ points on $X$.
Any isomorphism $g\colon X \overset{\sim}{\to} Y$ of smooth projective surfaces induces an isomorphism
%Any isomorphism $g : X \to Y$ of smooth projective surfaces induces an isomorphism
\[ g^{[n]}\colon X^{[n]} \overset{\sim}{\to} Y^{[n]}. \]
%\[ g^{[n]} : X^{[n]} \to Y^{[n]}. \]
%by sending a subscheme parametrized by $X^{[n]}$ to its image under $g$.
An isomorphism $\sigma \colon X^{[n]} \overset{\sim}{\to} Y^{[n]}$ is called \emph{natural} if $\sigma = g^{[n]}$ for some~$g$. %$g \in \Isom(X,Y)$.
%An isomorphism $f : X^{[n]} \to Y^{[n]}$ is called \emph{natural} if $f = g^{[n]}$ for some~$g$. %$g \in \Isom(X,Y)$.
In this paper we investigate which Hilbert schemes of points on surfaces have non-natural automorphisms and isomorphisms.

Consider the case of K3 surfaces.
By a result of Beauville, the Hilbert scheme of points of a K3 surface is a hyperk\"ahler variety \cite[Th\'eor\`eme 3]{Bea}.
Isomorphisms of hyperk\"ahler varieties can be controlled using the global Torelli theorem.
In particular, lattice arguments \cite{Zufetti} show that there exist non-isomorphic K3 surfaces $X_1$ and $X_2$ such that
$X_1^{[2]} \cong X_2^{[2]}$, see also \cite[Example~7.2]{Yoshioka}. % for a direct construction.
%In particular, lattice arguments show that there exists non-isomorphic K3 surfaces $S_1$ and $S_2$
%such that their Hilbert schemes are isomorphic: $S_1^{[2]} \cong S_2^{[2]}$.\footnote{
%For any $S_1$ very general K3 surface of degree $6 \cdot 73$ there exist a non-isomorphic K3 surface $S_2$, constructed as a moduli space of sheaves on $S_1$,
%such that $S^{[2]} \cong (S')^{[2]}$, see \cite[Ex.4.11]{Brakkee}.
%}
By construction these isomorphisms are not natural.
Similarly, the involution of the Hilbert schemes of $2$ points on a general quartic K3
that sends a subscheme to the residual subscheme of the line passing through it,
does not preserve the diagonal and is hence not natural, see Beauville \cite[\S6]{Bea2}.
The geometric construction and classification of auto- and isomorphisms of hyperk\"ahler varieties of $\mathrm{K3}^{[n]}$-type
is a rich and beautiful subject in its own right.

From now on we drop the condition on $X$ to be Calabi--Yau.
We first focus on the existence of non-natural automorphisms of $X^{[n]}$. % for an arbitrary surface $X$.
By a computation of Boissi\`ere \cite[Corollaire 1]{Bos}, the automorphism groups of $X^{[n]}$ and $X$ have the same dimension
and hence the same identity component.
The question of whether non-natural automorphisms exist
is therefore discrete in nature.

Our first result is the following.
Recall that a surface $X$ is called \emph{weak Fano} if $\omega_X^{-1}$ is nef and big.

\begin{thm} \label{thm1}
Let $X$ be a smooth projective surface which is weak Fano or of general type, and let $n$ be any integer.
Except for the case $(C_1 \times C_2)^{[2]}$, where $C_1$ and $C_2$ are smooth curves, every automorphism of $X^{[n]}$ is natural.
\end{thm}

%In \cite[Theorem 1.3(i) and (iii)]{hayashi} Hayashi has studied the case where $X$ is a rational surface such that the Iitaka dimension of $\omega_X^{-1}$ is at least 1. This is more general than our weak Fano assumption, but our methods also work for surfaces of general type which are not necessarily simply connected.

%2. Let $Y$ be a cubic fourfold with order $11$ automorphism. Then $F(Y)$ is birational to $X^{[2]}$ for a K3 surface $X$ and the
%automorphism on $Y$ induces a non-natural birational map of $X^{[2]}$ of order $11$, see Mongardi \cite[Sec.4.5]{MonPhd}.}

The second result deals with the case left open in the first theorem:

\begin{thm} \label{thm2} Let $C_1, C_2$ be smooth curves, either both rational or both of genus $g \geq 2$. %not both isomorphic to an elliptic curve.
Up to composing with natural automorphisms, there exists a unique non-natural automorphism of $(C_1 \times C_2)^{[2]}$.\end{thm}

The non-natural automorphism on $(C_1 \times C_2)^{[2]}$ can be described as follows.
On the complement of the diagonal it sends the cycle $(x_1, y_1) + (x_2,y_2)$ on $C_1 \times C_2$ to the cycle $(x_1, y_2) + (x_2,y_1)$.
Formally, it is defined by lifting the covering involution of
the natural map of symmetric products
\[ (C_1 \times C_2)^{(2)} \to C_1^{(2)} \times C_2^{(2)} \]
to the Hilbert scheme.

%In \cite[Theorem 1.3(ii)]{hayashi} Hayashi has exhibited this non-natural automorphism for the case $X=\mathbb{P}^1\times\mathbb{P}^1$, and shown its uniqueness up to composition with natural automorphisms.
%(The map $X^{(2)} \to C_1^{(2)} \times C_2^{(2)}$ is quasi-finite and proper, so finite, and by Miracle Flatness \cite[.23.1]{Mat} also flat. Hence its a double cover and there is a covering involution. The covering involution lifts to the Hilbert scheme by universal property.)

%For an Enriques surface $E$, the automorphisms of the Hilbert scheme of points have been studied by Hayashi \cite{Hay}.
%He finds that the Hilbert scheme $E^{[2]}$ admits no non-natural automorphisms \cite{Hay2}.
Boissi\`ere and Sarti proved that if $X$ is a K3 surface then an automorphism $f \in \Aut(X^{[n]})$
is natural if and only if it preserves the diagonal \cite{BS}.
By a result of Hayashi the same holds if $X$ is an Enriques surface \cite[Theorem 1.2]{Hay}.
This gives some evidence in favour of a positive answer to the following question.
\begin{question}
Suppose $X$ is a smooth projective surface and $\sigma \colon X^{[n]} \overset{\sim}{\to} X^{[n]}$ is an automorphism preserving the diagonal.
%Suppose $X$ is a smooth projective surface and $\sigma \colon X^{[n]} \to X^{[n]}$ is an automorphism preserving the diagonal.
Excluding the case $X = C_1 \times C_2$ and $n = 2$, does it follow that $\sigma$ is natural?
\end{question}

For a smooth projective curve $C$ of genus $g$ the Hilbert scheme $C^{[n]}$ is isomorphic to the symmetric product $C^{(n)}$.
In \cite{BG} Biswas and G\'omez show that if $g>2$ and $n>2g-2$ then every automorphism of the $n$-th symmetric product of $C$ is natural.
On the other hand non-natural automorphisms on $(\p^1)^{[n]} \cong \p^n$ for $n \geq 2$ are abundant.

If $X$ is smooth of dimension $\geq 3$ then the Hilbert scheme $X^{[n]}$ is smooth if and only if $n \leq 3$ \cite[Theorem 3.0.1]{MR1616606}. As a first step in understanding the situation
in these cases
we prove the following result.
\begin{thm} \label{Hilb2Pn}
Every automorphism of $(\p^n)^{[2]}$ is natural.
%$(\p^n)^{[2]}$ does not have any non-natural automorphisms.
\end{thm}

% is this a good phrasing?
The construction of the non-natural automorphism of $(C_1\times C_2)^{[2]}$ generalizes to products of higher dimensionsional varieties,
but we make no claims regarding the analogue of the uniqueness result of Theorem~\ref{thm2}.

Bondal and Orlov \cite[Theorem 2.5]{BO} proved that any derived equivalence between smooth projective varieties with one of them having ample or anti-ample canonical bundle
is induced by an isomorphism of the underlying varieties. We obtain the following Hilbert scheme analog. % of this reconstruction result.
%is induced by an isomorphism of the underlying varieties. We obtain the following Hilbert scheme analog.
\begin{cor} \label{Cor1}
Let $X,Y$ be smooth projective surfaces and assume that $Y$ is weak Fano or of general type.
Moreover if $Y$ is a product of curves assume $n \geq 3$.
Then for every isomorphism $\sigma \colon X^{[n]} \overset{\sim}{\to} Y^{[n]}$
there exist an isomorphism $g \colon X \overset{\sim}{\to} Y$ such that $\sigma=g^{[n]}$.
%Then for every isomorphism $\sigma : X^{[n]} \to Y^{[n]}$
%there exist an isomorphism $g : X \to Y$ such that $\sigma=g^{[n]}$.
\end{cor}
%The case of rational surfaces such that the Iitaka dimension of $\omega_X^{-1}$ is at least~1 is covered by \cite[Theorem 1.4]{hayashi}.

After a first version of this paper appeared online, Hayashi
made us aware of the preprint \cite{hayashi}
%pointed out to us the (not yet online available) preprint \cite{hayashi}
%After a first version of this paper appeared online we were made aware of a (not yet online available) preprint by Hayashi \cite{hayashi}
in which he proves 
Theorems~\ref{thm1} and~\ref{thm2} and Corollary~\ref{Cor1}
for rational surfaces such that the Iitaka dimension of $\omega_X^{-1}$ is at least $1$.
Hayashi's arguments do not apply to
surfaces with non-trivial fundamental group or in general type,
while our arguments do not apply
in Iitaka dimension $1$. For simply connected surfaces with $\omega_X^{ \pm 1}$ ample
the arguments are parallel, see Section~\ref{section_p2} for an outline of that case.
%When combining both arguments the general answer to Question~\ref{4etsdg} appears to be in reach.

An interesting question beyond the scope of this paper is to describe
%It is an interesting problem to determine
the group of derived auto-equivalences of the Hilbert scheme of points of a smooth projective surface.
In the weak Fano or general type case our results determine the group
of \emph{standard} auto-equivalences, that is the group generated by automorphisms of the variety, tensoring with line bundles, and shifts.
But by a result of Krug \cite[Theorem 1.1(ii)]{Krug}\footnote{The numbering refers to the (non yet publicly available) published version of the paper \cite{Krug}.
The corresponding results in the arXiv version are Theorem 1C and Conjecture~5.14.} there always exist non-standard auto-equivalences on the Hilbert scheme.
For Hilbert squares and Hilbert cubes of surfaces with ample or anti-ample canonical bundle
a proof of \cite[Conjecture 7.5]{Krug} combined with Theorems \ref{thm1} and \ref{thm2}
would give a full description of the derived auto-equivalence group.

%Another analogy with derived categories is in the description of the derived auto-equivalence group of a smooth projective variety with ample or anti-ample canonical bundle \cite[Theorem 3.1]{BO}. In these cases it says that it is comprised of only \emph{standard} auto-equivalences, namely the ones coming from automorphisms of the variety, tensoring with line bundles, and shifts. Our results can be used to give a description of the standard auto-equivalences, but by \cite[Theorem 1.1(ii)]{Krug} there always exist non-standard auto-equivalences for Hilbert schemes of points on surfaces. Note that a positive answer to \cite[Conjecture 7.5]{Krug} combined with Theorems \ref{thm1} and \ref{thm2} would give a full description of the derived auto-equivalence group of Hilbert squares and Hilbert cubes for surfaces with ample or anti-ample canonical bundle.

Hilbert schemes of points of Fano surfaces admit deformations which
may be understood as Hilbert schemes of non-commu\-tative deformations of Fano surfaces \cite{Li}.
It would be interesting to compare the automorphism groups of these deformations with the automorphism groups of the underlying non-commutative surfaces which were computed in \cite{BHH}, and see whether they are all natural in the appropriate sense.

\subsection{Acknowledgements}
Chiara Camere was part of the collaboration at an earlier stage and we are very grateful for her insights and input.
We would also like to thank Arnaud Beauville, Alberto Cattaneo, Daniel Huybrechts, Gebhard Martin and John Christian Ottem for their interest and useful discussions.
We thank Taro Hayashi for useful comments and for sending us the preprint \cite{hayashi}. % after the first version of our paper was posted to the arXiv.
The project originated at the Max Planck Institute for Mathematics in Bonn and we thank the institute for support.

\section{Preliminaries} \label{section_preliminaries}
%\subsection{Symmetric products}
Let $X$ be a smooth complex projective surface.
Let $X^{(n)}$ be the $n$-th symmetric product of $X$
obtained as the quotient of the cartesian product $X^n$ under the permutation action by the symmetric group $\mathrm{S}_n$.
Let $\rho\colon X^n \to X^{(n)}$ be the quotient map and let $p_i\colon X^n \to X$ be the projection onto the $i$-th factor.
Recall also the Hilbert--Chow morphism
\[ \epsilon \colon X^{[n]} \to X^{(n)} \]
which sends a subscheme $Z \subset X$ to its support.
The notation is summarized in the following diagram:
\[
\begin{tikzcd}
X^{[n]} \ar{dr}{\epsilon} & X^n \ar{r}{p_i} \ar{d}{\rho} & X \\
& X^{(n)} &
\end{tikzcd}
\]

For any line bundle $\CL$ on $X$ the tensor product $\bigotimes_{i=1}^n p_i^{\ast} \CL$
has a natural $\mathrm{S}_n$-invariant structure, and taking $\mathrm{S}_n$-invariants defines a line bundle $\CL_{(n)}$ on $X^{(n)}$.
If $\CL$ is (very) ample, then $\CL_{(n)}$ is (very) ample as well.
%\footnote{
%Proof: By definition of the quotient essentially.
%%If $\CL$ is very ample, then it defines an embedding $X \hookrightarrow \p^n$ for some $n$, and $\otimes_{i=1}^n p_i^{\ast} \CL$ defines an embedding $X^n \to \p^N$, so ... ?
%}
We also define the pullback to the Hilbert scheme:
\[ \CL_{[n]} \coloneqq \epsilon^{\ast} \CL_{(n)} \]
By arguments parallel to \cite[\S6]{Bea}, the canonical bundle of $X^{[n]}$ is
\[ \omega_{X^{[n]}} = (\omega_{X})_{[n]}. \]

The symmetric product $X^{(n)}$ is singular precisely at the diagonal $\Delta$ of cycles
supported at less than $n$ points.
% singularities along the diagonal.
By \cite[Theorem~18.18]{MS} the tangent space at $n x \in X^{(n)}$ for any $x \in X$ satisfies
\[ \dim \mathrm{T}_{X^{(n)}, nx} = \frac{n(n + 3)}{2}. \]
This shows that the small diagonal $\Delta_{\text{small}} = \{ n x \mid x \in X \}$ is distinguished in the symmetric product as the locus of points in $X^{(n)}$
where the Zariski tangent space is of maximal dimension.
%Whenever $x_i \in X$ are distinct points, then we have
%\[ T_{X^{(n)}, \sum_i m_i x_i} = \bigoplus_i T_{X^{(m_i)}, m_i x_i}. \]

%Let $\lambda = (\lambda_1, \ldots, \lambda_{\ell})$ be a partition of $n$. Consider the locus
%\[ \Delta_{\lambda} = \{ \sum_{i=1}^{\ell} \lambda_i x_i \, | \, x_1, \ldots, x_{\ell} \in X \}. \]
%On an open subset $\Delta_{\lambda}$ parametrizes the cycles of type $\lambda$.
%For example the small diagonal is
%\[ \Delta = \Delta_{(n)} = \{ n x | x \in X \} \subset X^{(n)}. \]
%By the discussion before, $\Delta$ is distinguished as the locus of points where the dimension of the tangent space is maximal.
%Similarly, if $n \geq 2$, the locus $\Delta_{(n-1,1)}$ is distinguished as the locus of points where the dimension of the tangent space is at least second-highest:
%\[ \Delta_{(n-1,1)} = \left\{ z \in X^{(n)} \,\middle|\, \dim T_{z, X^{(n)}} \geq \frac{(n-1)(n+2)}{2} + 2 \right\}. \]
%
%In case $n \geq 2$ consider the isomorphisms
%\begin{equation} \iota : X \to \Delta, \quad j : X \times X \to \Delta_{(n-1,1)} \label{ijmaps} \end{equation}
%given by $x \mapsto nx$ and $(x,y) \mapsto (n-1)x+y$.
%%\begin{align*}
%%\iota & : X \to \Delta, x \mapsto nx \\
%%j & : X \times X \to \Delta_{(n-1,1)}, (x,y) \mapsto (n-1)x + y.
%%\end{align*}
%By factoring $\iota$ and $j$ through the projection $\rho : X^n \to X^{(n)}$ one finds
%\begin{equation} \iota^{\ast} \CL_{(n)} = \CL^{\otimes n}, \quad j^{\ast} \CL_{(n)} = p_1^{\ast}(\CL^{\otimes (n-1)}) \otimes p_2^{\ast} \CL. \label{linebundlepullback} \end{equation}

For future use, we record the following lemma.
\begin{lemma}
\label{thm:DescentLemma}
Let $f \colon X \to Y$ be a morphism of projective varieties, where $Y$ is normal and $f$ has connected fibres.
Let $\mathcal{L}$ be an ample line bundle on $Y$.
For any automorphism $\sigma \colon X \to X$ such that $\sigma^* f^*\mathcal{L} \cong f^*\mathcal{L}$ there exists an isomorphism $\tau \colon Y \overset{\sim}{\to} Y$ such that $\tau \circ f = f \circ \sigma$.
\end{lemma}
\begin{proof}
%The assumptions imply that $f_*(\CO_X) = \CO_Y$, %and so $H^0(X, f^*L^{\otimes m}) = H^0(Y, L^{\otimes m})$.
%Indeed: Stein factorization, then if g : Y' -> Y is finite and Y is normal, then Y is reduced so g is generically flat and since f has connected fiber hence g is of degree 1 generically so birational. Then Zariski main theorem gives g is an isomorphism, so f O-connected.
%The normalization f : P^1 -> C of a singular cubic shows the statement is false without the normality condition on the target. Any automorphism of P^1 acts trivial on Pic, so if we take an automorphism that does not preserve the node points it does not descend.
By Stein factorization and our assumptions we have $f_*(\CO_X) = \CO_Y$, and so $\mathrm{H}^0(X, f^*\mathcal{L}^{\otimes m}) = \mathrm{H}^0(Y, \mathcal{L}^{\otimes m})$ for all $m\geq 0$. Applying the Proj construction to the corresponding graded algebra gives the isomorphism~$\tau$, and by construction~$\tau\circ f=f\circ\sigma$.
\end{proof}

\section{The basic strategy} \label{section_p2}
We first explain the proof of Theorem~\ref{thm1} %in the case where the canonical class of $X$ is ample or anti-ample and $X$ is simply connected.
under the assumption that
\begin{itemize}
% is this a good phrasing?
%\item $\omega_X$ or $\omega_X^{-1}$ is ample (in what follows we will write things for the case when $\omega_X$, the argument works verbatim for the case $\omega_X^{-1}$ ample),
\item $\omega_X$ or $\omega_X^{-1}$ is ample,
\item $X$ is simply connected,
\item $X$ is not a product of curves.
\end{itemize}
Let
$\sigma \colon X^{[n]} \overset{\sim}{\to} X^{[n]}$
%$\sigma \colon X^{[n]} \to X^{[n]}$
be an automorphism.
Since the differential of $\sigma$ is everywhere invertible we have
\[ \sigma^{\ast} \omega_{X^{[n]}} \cong \omega_{X^{[n]}}. \]

\vspace{1pt}
\noindent \emph{Step 1.} (Reduction to the symmetric product)
Because %explained in Section~\ref{section_preliminaries} the line bundle
$\omega_{X^{[n]}}$ is
%the pullback of an ample line bundle from the symmetric product $X^{(n)}$, by Lemma \ref{thm:DescentLemma} there exists an automorphism $\tau \colon X^{(n)} \to X^{(n)}$ which makes the following diagram commute:
the pullback of an ample or anti-ample line bundle from the symmetric product $X^{(n)}$, by Lemma \ref{thm:DescentLemma} there exists an automorphism $\tau \colon X^{(n)} \to X^{(n)}$ which makes the following diagram commute:
\[
\begin{tikzcd}
X^{[n]} \ar{d}{\epsilon} \ar{r}{\sigma} & X^{[n]}\ar{d}{\epsilon} \\
X^{(n)} \ar{r}{\tau} &  X^{(n)}.
\end{tikzcd}
\]
Since $\epsilon$ is birational, $\tau$ is the identity if and only if $\sigma$ is the identity.
We are hence reduced to studying automorphisms of the symmetric product.

\vspace{5pt}
\noindent \emph{Step 2.} (Lifting)
%Since $\tau$ it is an automorphism the tangent spaces at a point $z$ and its image $\tau(z)$ have the same dimension.
Since $\tau$ preserves the singular points the diagonal on the symmetric product is preserved:
\[ \tau(\Delta) = \Delta. \]
Let $D \subset X^n$ denote the big diagonal in $X^n$ and consider the restriction of the quotient map
\[ \rho_D \colon X^n \setminus D \, \to \, X^{(n)} \setminus \Delta. \]
Since $X$ is assumed to be simply connected and $D \subset X^n$ is of codimension $2$ in a smooth ambient space,
$X^n \setminus D$ is also simply connected.
Hence $\rho_D$ is the universal covering space of $X^{(n)} \setminus \Delta$.
Applying the universal lifting property to the morphism $\tau \circ \rho_D$ we obtain an
automorphism $f \in \Aut( X^n \setminus D)$ which makes the following diagram commute:
\[
\begin{tikzcd}
X^n \setminus D \ar{d}{\rho} \ar{r}{f} & X^n \setminus D \ar{d}{\rho} \\
X^{(n)} \setminus \Delta \ar{r}{\tau} &  X^{(n)} \setminus \Delta.
\end{tikzcd}
\]

\vspace{5pt}
\noindent\emph{Step 3.} (Extension to $X^n$)
Since $D$ is of codimension $2$ in $X^n$ and since $X^n$ is smooth hence normal,
every section of a line bundle on the complement of $D$ extends.
Applying this to $\omega_{X^n}$ and its powers, the automorphism $f$ induces a graded ring automorphism of
%\[ \bigoplus_{m\geq 0} \mathrm{H}^0(X^n, \omega_{X^n}^{\otimes m}). \]
%Since $\omega_X$ is ample, this in turn induces an automorphism of $X^n$ that extends $f$.
\[ \bigoplus_{m} \mathrm{H}^0(X^n, \omega_{X^n}^{\otimes m}). \]
Since $\omega_X$ is ample or anti-ample, this in turn induces an automorphism of $X^n$ that extends $f$.
We will denote this extension by $f$ as well. We have constructed an automorphism
\[ f \colon X^n \overset{\sim}{\to} X^n \]
%\[ f \colon X^n \to X^n \]
such that $\rho \circ f = \tau \circ \rho$.

\vspace{5pt}
\noindent\emph{Step 4.} (Splitting the automorphism)
Let $\CL$ be a very ample line bundle on $X$ and let
\[ \CL_i \coloneqq p_i^{\ast} \CL \]
be its pullback to $X^n$ along the projection to the $i$-th factor.
The projection $p_i$ is hence the morphism associated to the complete linear system of $\CL_i$.
We following the arguments of \cite[Theorem~4.1]{Oguiso} and consider $f^{\ast} \CL_i$.

Since $X$ is simply connected, we have $\mathrm{H}^1(X,\BZ) = 0$. On the one hand
this implies that the $\Pic^0(X^n) = \Pic^0(X)^n = 0$.
On the other hand, we have $\mathrm{H}^2(X^n, \BZ) = \mathrm{H}^2(X,\BZ)^{\oplus n}$ and hence
\[ \Pic(X^n) \cong \mathrm{H}^{1,1}(X^n, \BZ) \cong \mathrm{H}^{1,1}(X,\BZ)^{\oplus n} \cong \Pic(X)^{\oplus n}. \]

We conclude that there exist line bundles $\CM_j$ on $X$ such that
\[ f^{\ast} \CL_i \cong p_1^{\ast} \CM_1 \otimes \cdots \otimes p_n^{\ast} \CM_n. \]
This implies that
$p_i \circ f = h \circ (g_1 \times \cdots \times g_n)$
for some morphisms $g_j \colon X \to Y_j$ and some isomorphism $h \colon \prod_i Y_i \xrightarrow{\sim} X$.
Since $X$ is not a product of curves, one of the $g_j$ is an isomorphism and the others are projections to a point.
Hence
$p_i \circ f$ only depends on the corresponding $j(i)$-th component.
After composing $f$ with the permutation that sends $i$ to $j(i)$ we therefore get
\[ f = f_1 \times \cdots \times f_n \]
for some $f_i \in \Aut (X)$.

The automorphism $\tau \in \Aut (X^{(n)})$ preserves the small diagonal as the locus where the tangent space has maximal dimension.
%The automorphism $\tau \in \Aut (X^{(n)})$ preserves the small diagonal as the locus where the tangent space has maximal rank.
It follows that $f$ preserves the small diagonal in $X^n$,
so all the $f_i$ are the same, and hence that $\sigma$ is natural. %yields $f_1 = \ldots = f_n$ and so $\sigma = f_1^{[n]}$ as desired.

\section{The general case}
We present the proof of the main theorem.
We proceed as in Section~\ref{section_p2} but at every step we need to find an
argument that works for weak Fano surfaces and for surfaces of general type. %in greater generality.
Our assumption throughout is that $X$ is a smooth complex projective surface.

For a weak Fano surface $X$ we will use that $\omega_X^{-1}$ is semiample, that is a power of it is basepoint free.
Moreover, the morphism defined by the linear system of $\omega_X^{- \otimes m}$ is birational for appropriate $m \gg 0$,
%; its image is the anti-canonical model of $X$,
see \cite[Theorem~4.3]{Sakai}.
%The image is called the anti-canonical model.
%the linear system of a high enough power of it defines a birational morphism onto its image.

\subsection{Reduction to symmetric product}
We begin by giving a criterion for when an automorphism of the Hilbert scheme descends to the symmetric product.
Let $\alpha \in \mathrm{H}_2(X^{[n]}, \BZ)$ be the class of a $\p^1$-fiber of the map $\epsilon \colon X^{[n]} \to X^{(n)}$.
In particular, $\alpha$ is the unique primitive curve class such that
\[ \epsilon_{\ast}\alpha = 0. \]
%\[ \epsilon_{\ast} \alpha = 0. \]

\begin{prop} \label{Prop_Sym1}
Let $\sigma \in \Aut (X^{[n]})$. If
$\sigma_{\ast} \alpha = \alpha$
%$\sigma_{\ast} \alpha = \alpha$
then there exists an automorphism $\tau \in \Aut(X^{(n)})$ such that the following diagram commutes:
%such that $\tau \circ \epsilon = \epsilon \circ \sigma$.
%% which makes the following diagram commute:
\[
\begin{tikzcd}
X^{[n]} \ar{d}{\epsilon} \ar{r}{\sigma} & X^{[n]}\ar{d}{\epsilon} \\
X^{(n)} \ar{r}{\tau} &  X^{(n)}.
\end{tikzcd}
\]
\end{prop}
\begin{proof}
We first show that the morphism $\epsilon \colon X^{[n]} \to X^{(n)}$ is the initial object in the category of morphisms $f \colon X^{[n]} \to Z$, where $Z$ is a projective scheme and $f$ contracts the $\p^1$-fibers of $\epsilon$ (or equivalently, $f_{\ast} \alpha  = 0$).

Indeed, let $f \colon X^{[n]} \to Z$ be such a morphism and consider the scheme
\[ Y = (\epsilon \times f)(X^{[n]}) \subset X^{(n)} \times Z. \]
We claim the projection to the first factor $p \colon Y \to X^{(n)}$ is an isomorphism.
Since $X^{(n)}$ is normal (as the quotient of the normal space $X^n$ by a finite group) and $p$ is birational and proper,
by Zariski's main theorem it suffices to show that $p$ is finite.
If $p$ is not finite, it contracts a curve $\Sigma$.
Then there exists a curve $\Sigma' \subset X^{[n]}$ such that its image under $\epsilon \times f$ is $\Sigma$
(for example, take the preimage of $\Sigma$ and if that is of dimension $>1$ cut it down by sections of a relatively ample class).
Since by assumption we have $(\epsilon \times f)_{\ast} \alpha = 0$, the class of $\Sigma'$ is linearly independent (over $\BQ$) of $\alpha$ and
contracted by $\epsilon = p \circ (\epsilon \times f)$.
%$p_{\ast} (\epsilon \times f)_{\ast} \beta = \epsilon_{\ast} \beta = 0$.
But the kernel of
\[ \epsilon_{\ast} \colon \mathrm{H}_2(X^{[n]}, \BQ) \to \mathrm{H}_2(X^{(n)}, \BQ) \]
is $1$-dimensional and spanned by $\alpha$ which gives a contradiction.
We conclude that $p$ is an isomorphism and hence that $Y$ is the graph of a morphism $g \colon X^{(n)} \to Z$ with $g \circ \epsilon = f$.
Since $\epsilon$ is birational, $g$ is unique.

Applying the above universal property of $\epsilon$ to $\epsilon \circ \sigma$
we obtain a morphism $\tau \colon X^{(n)} \to X^{(n)}$ such that $\epsilon \circ \sigma = \tau \circ \epsilon$.
On the other hand, the same argument also implies that $\epsilon \circ \sigma$ is initial, so $\tau$ is an isomorphism.
\end{proof}

We apply our criterion to the case at hand:

\begin{prop} \label{Prop_Sym2}
Let $X$ be a smooth projective surface which is weak Fano or of general type. Then for every automorphism $\sigma \in \Aut(X^{[n]})$
we have $\sigma_{\ast} \alpha = \alpha$.
%we have $\sigma_{\ast} \alpha = \alpha$.
\end{prop}
\begin{proof}
We assume that $X$ is of general type. The weak Fano case is parallel.
Let $Y$ be the canonical model of $X$, which is a surface with isolated singular points. % which is a surface of dimension $2$ with singular points $y_1, \ldots, y_r \in Y$.
The canonical line bundle on $X^{[n]}$ induces a morphism
\[ \varphi \colon X^{[n]} \to Y^{(n)} \]
and since $\sigma$ preserves this line bundle,
by Lemma~\ref{thm:DescentLemma},
there exists $\tau \in \Aut(Y^{(n)})$ such that
\[
\begin{tikzcd}
X^{[n]} \ar{d}{\varphi} \ar{r}{\sigma} & X^{[n]}\ar{d}{\varphi} \\
Y^{(n)} \ar{r}{\tau} &  Y^{(n)}.
\end{tikzcd}
\]
commutes.

Let $y_1, \ldots, y_r$ be the singular points of $Y$. The singular locus of $Y^{(n)}$ is
\begin{equation} \mathrm{Sing}\, Y^{(n)} = \Delta \cup D_{y_1} \cup \ldots \cup D_{y_r} \label{rgdg} \end{equation}
where for a point $y \in Y$ the subscheme $D_y \subset Y^{(n)}$ is defined by
\begin{equation} D_y = \{ y + z \mid z \in Y^{(n-1)} \}. \label{3452354} \end{equation}
\noindent \emph{Claim.} The automorphism $\tau$ preserves the diagonal $\Delta \subset Y^{(n)}$.\\
\emph{Proof of the claim:} The automorphism $\tau$ preserves the singular locus of $Y^{(n)}$.
Since \eqref{rgdg} is the decomposition of the singular locus into irreducible components we need to exclude the case that $\tau(\Delta) = D_{y_i}$ for some $i$.
If $n \geq 3$ then the normalizations of $D_y$ and $\Delta$ are
\begin{equation} \tilde{D_y} = D_y \cong Y^{(n-1)}, \quad \tilde{\Delta} = Y \times Y^{(n-2)}. \label{kggh} \end{equation}
%\[ \tilde{D_y} = D_y \cong Y^{(n-1)}, \quad \tilde{\Delta} = Y \times Y^{(n-2)}. \]
To see this for the diagonal, we have a natural finite birational map $Y \times Y^{(n-2)} \to \Delta$.
Since the source is normal it factors through a map to the normalization, which is an isomorphism by Zariski's main theorem.
Since $Y^{(n-1)}$ and $Y \times Y^{(n-2)}$ are not isomorphic for $n \geq 3$ this completes the claim.
In case $n=2$ both $\Delta$ and $D_{y_i}$ are isomorphic to $Y$. The corresponding inclusion maps factor as
\begin{align*}
\iota_{\Delta} & \colon Y \xrightarrow{\Delta} D \subset Y \times Y \to Y^{(2)} \\
\iota_{D_j} & \colon Y \cong y_i \times Y \subset Y^2 \to Y^{(2)}.
\end{align*}
From this we get
\[ \iota_{\Delta}^{\ast} \omega_{Y^{(2)}} = \omega_Y^{\otimes 2}, \quad \iota_{D_j}^{\ast} \omega_{Y^{(2)}} = \omega_Y. \]
Since $\omega_{Y^{(2)}}$ is preserved under pullback by $\tau$, this excludes $\tau(\Delta) = D_{x_j}$. \qed

\vspace{5pt}
We return to the proof of the proposition.
Let $E \subset X^{[n]}$ denote the exceptional divisor.
Let $C_{y_i} \subset X$ be the curve contracted
to $y_i$
under the canonical map $X \to Y$ % to $y_i$
and let
$V_i \subset X^{[n]}$ be the preimage under $\epsilon$ of the subscheme
\[ \{ w_1 + w_2 \mid w_1 \in (C_{y_i})^{(2)}, w_2 \in X^{(n-2)} \} \subset X^{(n)}. \]
Then
\[ \varphi^{-1}(\Delta) = E \cup \bigcup_{i=1}^{r} V_i. \]
By the claim $\sigma$ preserves $\varphi^{-1}(\Delta)$.
Since every $V_i$ is of codimension $\geq 2$ while $E$ is a divisor,
we conclude $\sigma(E) = E$. So we get a commutative diagram
\[
\begin{tikzcd}
E \ar{d}{\varphi} \ar{r}{\sigma} & E\ar{d}{\varphi} \\
\Delta \ar{r}{\tau} & \Delta.
\end{tikzcd}
\]
In particular, $\sigma$ sends fibers of $\varphi$ to fibers.
Since the generic fiber of $E \to \Delta$ is precisely the $\p^1$ contracted by $\epsilon$ we are done.
\end{proof}

\subsection{Lifting}
Let $X$ be a smooth projective surface.
We show that every automorphism of $X^{(n)}$ lifts to an automorphism of $X^n \setminus D$.
\begin{prop} \label{prop_lifting}
For every $\tau \in \Aut(X^{(n)})$ there exists $f \in \Aut( X^n \setminus D )$ such that
the following diagram commutes:
\[
\begin{tikzcd}
X^n \setminus D \ar{d}{\rho} \ar{r}{f} & X^{n} \setminus D \ar{d}{\rho} \\
X^{(n)} \setminus \Delta \ar{r}{\tau} &  X^{(n)} \setminus \Delta.
\end{tikzcd}
\]
\end{prop}
\begin{proof}
  The main idea is that since $\rho$ is a normal covering space, by the standard lifting criterion we only have to show
  \[
    (\tau \circ \rho)_{*}(\pi_{1}(X^{(n)} \setminus \Delta)) \subset \rho_{*}(\pi_{1}(X^{n} \setminus \Delta)).
  \]

  We first make a simplification:
  Since the small diagonal is preserved by $\tau$ and isomorphic to $X$
  the restriction $\tau|_{\Delta_{\text{small}}}$ defines an automorphism $g \in \Aut(X)$.
  Replacing $\tau$ by $(g^{-1})^{(n)} \circ \tau$ we may assume that
  \begin{equation*} \tau|_{\Delta_{\text{small}}} = \mathrm{id}_{\Delta_{\text{small}}}. \label{352t} \end{equation*}

  Choose a point $x \in X$, a small open ball $U \subset X$ with $x \in U$, and $n$ distinct points $x_{1}, \ldots, x_{n} \in U \setminus \{x\}$.
  Let $G = \pi_{1}(X, x)$, and note that we have canonical identifications $G \cong \pi_{1}(X, x_{i})$ for every $i$, given by connecting $x$ to $x_{i}$ via a path in $U$.
  Then
  \[ \pi_1(X^n \setminus D, (x_{i})) = \pi_1(X^n, (x_{i})) = G^n. \]

  The map
  \[
    \rho \colon X^n \setminus D \to X^{(n)} \setminus \Delta
  \]
  is obtained by taking the quotient by the free action of $\mathrm{S}_n$, hence it is a normal covering space, and we have the exact sequence of groups
  \begin{equation} 1 \to \pi_{1}(X^{n} \setminus D, (x_{i})) \to \pi_1(X^{(n)} \setminus D, \textstyle \sum x_{i}) \to \mathrm{S}_n \to 1. \label{ses} \end{equation}

  We define a splitting of this short exact sequence as follows.
  The inclusion of $U$ in $X$ induces an inclusion
  \[ U^{(n)} \setminus \Delta \hookrightarrow X^{(n)} \setminus \Delta. \]
  By simply-connectedness of $U$ we get $\mathrm{S}_{n} \cong \pi_1(U^{(n)} \setminus \Delta, \textstyle\sum x_{i})$, and the inclusion
  \[ \mathrm{S}_n \cong \pi_1(U^{(n)} \setminus \Delta, \textstyle\sum x_{i}) \to \pi_1( X^{(n)} \setminus \Delta, \textstyle\sum x_{i}) \]
  splits \eqref{ses}.
  Thus we have
  \[
    \pi_1( X^{(n)} \setminus \Delta, \textstyle\sum x_{i} ) \cong G^n \rtimes \mathrm{S}_n,
  \]
  and one can check that the conjugation action of $\mathrm{S}_n$ on $G^n$ is the standard permutation of factors.

  The set $\tau(U^{(n)})$ is an open neighbourhood of $nx \in X^{(n)}$.
  Picking some sufficiently small open ball $V$ with $x \in V \subset U$, we have $V^{(n)} \subset U^{(n)} \cap \tau(U^{(n)})$.
  We may assume that $x_{i} \in V$ for all $i$, hence $\sum x_{i} \in V^{(n)} \setminus \Delta$.
  We have the commutative diagram
  \[
    \begin{tikzcd}
      \pi_{1}(V^{(n)} \setminus \Delta, \textstyle \sum x_{i}) \arrow[r, "a"] \arrow[d, "b"] & \pi_{1}(U^{(n)} \setminus \Delta, \textstyle \sum x_{i}) \arrow[d, "c"] \\
      \pi_{1}(\tau(U^{(n)}) \setminus \Delta, \textstyle \sum x_{i}) \arrow[r, "d"] & \pi_{1}(X^{(n)} \setminus \Delta, \textstyle \sum x_{i}),
    \end{tikzcd}
  \]
  where the three upper-left groups are isomorphic to $\mathrm{S}_{n}$ and where $c \circ a = d \circ b$ is injective.
  It follows that $a$ and $b$ are isomorphisms and that the three upper-left groups are equal as subgroups of $\pi_{1}(X^{(n)} \setminus \Delta, \sum x_{i})$.

  Picking a path in $\tau(U^{(n)}) \setminus \Delta$ from $\tau(\sum x_{i})$ to $\sum x_{i}$ gives isomorphisms
  \begin{gather*}
    s \colon \pi_{1}(X^{(n)} \setminus \Delta, \tau(\textstyle\sum x_{i})) \stackrel{\sim}{\to} \pi_{1}(X^{(n)} \setminus \Delta, \textstyle\sum x_{i}) \\
    s \colon \pi_{1}(\tau(U^{(n)}) \setminus \Delta, \tau(\textstyle\sum x_{i})) \stackrel{\sim}{\to}
    \pi_{1}(\tau(U^{(n)}) \setminus \Delta, \textstyle\sum x_{i})
    %\pi_{1}(U^{(n)} \setminus \Delta, \textstyle\sum x_{i}),
  \end{gather*}
  We have the equality of subgroups of $\pi_{1}(X^{(n)} \setminus \Delta, \sum x_{i})$
  \begin{align*}
    (s\circ\tau_{*})(\pi_{1}(U^{(n)}\setminus \Delta), \textstyle \sum x_{i})
    &= s(\pi_{1}(\tau(U^{(n)}) \setminus \Delta, \tau(\textstyle \sum x_{i}))) \\
    &= \pi_{1}(\tau(U^{(n)}) \setminus \Delta, \textstyle \sum x_{i}) \\
    &= \pi_{1}(U^{(n)} \setminus \Delta, \textstyle \sum x_{i}).
  \end{align*}
  Therefore, in the presentation $\pi_{1}(X^{(n)} \setminus \Delta, \sum x_{i}) = G^{n} \rtimes \mathrm{S}_{n}$ and the notation of \S\ref{subsection:some-group-theory}, we have
  \[
    (s\circ\tau_{*})(e_{G^{n}} \rtimes \mathrm{S}_{n}) = e_{G^{n}} \rtimes \mathrm{S}_{n},
  \]
  and so by Lemma \ref{thm:GroupTheoryLemma}, we have
  \[
    (s\circ\tau_{*})(G^{n} \rtimes e_{\mathrm{S}_{n}}) = G^{n} \rtimes e_{\mathrm{S}_{n}}.
  \]

  Since $G^{n} \rtimes e_{\mathrm{S}_{n}} = \rho_{*}(\pi_{1}(X^{n} \setminus \Delta), (x_{i}))$, this implies
  \[
    \tau_{*}(\rho_{*}(\pi_{1}(X^{n} \setminus \Delta), (x_{i}))) = s^{-1}(\rho_{*}(\pi_{1}(X^{n} \setminus \Delta), (x_{i}))) = \rho_{*}(\pi_{1}(X^{n} \setminus \Delta, y)),
  \]
  where $y \in X^{n}$ satisfies $\rho(y) = \tau(\sum x_{i})$, and is the parallel transport of $(x_{i})$ along the path defining $s$.
  By the lifting criterion for covering spaces, it follows that $\tau \circ \rho$ lifts to an
  automorphism
  $f$ as required.
  \end{proof}

\subsection{Some group theory}
\label{subsection:some-group-theory}
\begin{lemma} \label{Lemma_XYZ}
  Let $n \ge 3$, and let $\sigma \in \mathrm{S}_{n}$ be such that $\sigma$ commutes with all its conjugates, and such that the centraliser $\mathrm{C}(\sigma)$ contains a subgroup isomorphic to $\mathrm{S}_{n-1}$.
  Then $\sigma = e_{\mathrm{S}_{n}}$.
\end{lemma}
\begin{proof}
  By the first assumption on $\sigma$, the elements $g\sigma g^{-1}, g \in \mathrm{S}_{n}$ generate a normal, abelian subgroup $H$ of $\mathrm{S}_{n}$.
  If $n \ge 5$, then $\mathrm{A}_n$ is the only non-trivial normal subgroup of $\mathrm{S}_n$ (\cite[\S10.8.8]{Scott}).
  As $H$ is abelian, it must therefore be trivial, and so $\sigma = e_{\mathrm{S}_{n}}$.
  The remaining cases $n = 3,4$ are checked directly.
\end{proof}

Let $G$ be a group, and define $G^{n} \rtimes \mathrm{S}_{n}$ by the permutation action of $\mathrm{S}_{n}$ on the factors of $G^{n}$.
% P: do we need to define a semidirect product?
%, i.e.~
%\begin{gather*}
%  G^{n} \rtimes \mathrm{S}_{n} = \{((g_{i})_{i= 1}^{n}, \sigma) \mid g_{i} \in G, \sigma \in \mathrm{S}_{n}\} \\
%  ((g_{i})_{i=1}^{n}, \sigma)((g'_{i})_{i=1}^{n}, \sigma') = ((g_{i}g_{\sigma(i)})_{i=1}^{n}, \sigma\sigma').
%\end{gather*}
We write $e\rtimes\mathrm{S}_n=e_{G^n} \rtimes \mathrm{S}_{n}$ and $G^{n} \rtimes e=G^n\rtimes e_{\mathrm{S}_n} \subset G^{n} \rtimes \mathrm{S}_{n}$ for the groups $\mathrm{S}_{n}, G^{n}$ thought of as subgroups of $G^{n} \rtimes \mathrm{S}_{n}$.
%We write $e \rtimes \mathrm{S}_{n}, G^{n} \rtimes e \subset G^{n} \rtimes \mathrm{S}_{n}$ for the groups $\mathrm{S}_{n}, G^{n}$ thought of as subgroups of $G^{n} \rtimes \mathrm{S}_{n}$.
\begin{lemma}
\label{thm:GroupTheoryLemma}
  Let $\tau \colon G^{n} \rtimes \mathrm{S}_{n} \xrightarrow{\sim} G^{n} \rtimes \mathrm{S}_{n}$ be an automorphism.
  If $\tau(e \rtimes \mathrm{S}_{n}) = e \rtimes \mathrm{S}_{n}$, then $\tau(G^{n} \rtimes e) = G^{n} \rtimes e$.
\end{lemma}
\begin{proof}
  We need only show that $\tau(G^{n} \rtimes e) \subseteq G^{n} \rtimes e$, since applying this to $\tau^{-1}$ gives
  \[
    \tau^{-1}(G^{n} \rtimes e) \subseteq G^{n}\rtimes e \Rightarrow G \rtimes e \subseteq \tau(G^{n} \rtimes e).
  \]
  Let $g \in G$ be any element, and let $g_{(i)} \in G^{n}$ be the inclusion of $g$ in the $i$-th factor.
  The group $G^{n} \rtimes e$ is generated by elements of the form $(g_{(i)}, e_{\mathrm{S}_{n}})$, hence it suffices to show that $\tau(g_{(i)}, e_{\mathrm{S}_{n}}) \in G^{n} \rtimes e$.

  Note first that $(g_{(i)}, e_{\mathrm{S}_{n}})$ commutes with all its $e \rtimes \mathrm{S}_{n}$-conjugates, hence so does $\tau(g_{(i)}, e_{\mathrm{S}_{n}})$.
  Note further that $(g_{(i)}, e_{\mathrm{S}_{n}})$ commutes with every element $(e_{G^{n}}, \delta)$ where $\delta$ fixes $i$.
  Hence the centraliser of $(g_{(i)},e_{\mathrm{S}_{n}})$ contains a subgroup isomorphic to $\mathrm{S}_{n-1}$ inside $e \rtimes \mathrm{S}_{n}$, and so likewise the centraliser of $\tau(g_{(i)},e_{\mathrm{S}_{n}})$ contains such a subgroup in $e \rtimes \mathrm{S}_{n}$.
  If now
  \[
    \tau(g_{(i)}, e_{\mathrm{S}_{n}}) = (x,\sigma), \ \ \ \ x \in G^{n}, \sigma \in \mathrm{S}_{n},
  \]
  we have that $\sigma$ commutes with all its conjugates, and $|\mathrm{C}_{\mathrm{S}_{n}}(\sigma)| \ge (n-1)!$, whence by Lemma~\ref{Lemma_XYZ} we have $\sigma = e_{\mathrm{S}_{n}}$ if $n \ge 3$.

  Assume that $n = 2$.
  Let $\delta \in \mathrm{S}_2$ be the non-trivial element.
  By our assumption, $\tau(e_{G^2},\delta) = (e_{G^2},\delta)$, and so the centraliser $C = \mathrm{C}_{G^{2} \rtimes \mathrm{S}_{2}}(e_{G^2},\delta)$ is preserved by $\tau$.
By a direct check we have
%It is easy to see that
$C = \{((g,g), e_{\mathrm{S}_{2}}), ((g,g),\delta) \mid g \in G\}$.
  Next observe that elements $((g,g), e_{\mathrm{S}_{2}})$ can all be written as a product of an element $x$ and its $(e_{G^2},\delta)$-conjugate, e.g.~take $x = ((g,e_G),e_{\mathrm{S}_{2}})$.
  On the other hand, elements of the form $((g,g), \delta)$ cannot be written in such a way, since the product of an element with its $(e_{G^2},\delta)$-conjugate must have $e_{\mathrm{S}_{2}}$ in the $\mathrm{S}_{2}$-factor.
  Therefore the set of elements of the form $((g,g),e_{\mathrm{S}_2})$ is preserved by $\tau$.

  Since these elements form a subgroup of $G^2 \rtimes \mathrm{S}_{2}$, the automorphism $\tau$ defines by restriction an automorphism $\psi$ of $G$.
  After post-composing $\tau$ with the automorphism $((g,h),x) \mapsto ((\psi^{-1}(g),\psi^{-1}(h)), x)$, we may assume that $\tau$ in fact fixes each element $((g,g),e_{\mathrm{S}_2})$.

  Now consider the element $((g,e_{G}),e_{\mathrm{S}_{2}})$. It satisfies the following equation in $x$:
  \begin{equation}
    \label{eqn:GroupEquation}
    x(e_{G^{2}},\delta)x(e_{G^{2}},\delta) = ((g,g),e_{\mathrm{S}_2}).
  \end{equation}
  The same equation must therefore be satisfied by $\tau((g,e_{G}),e_{\mathrm{S}_{2}})$.

  Assume now for a contradiction that $\tau((g,e_{G}),e_{\mathrm{S}_{2}}) = ((h,i),\delta)$.
  Since equation \eqref{eqn:GroupEquation} is satisfied by $((h,i),\delta)$, we must have $(h^2, i^2) = (g,g)$.
  Therefore $g$ is a square, hence so is $((g,e_{G}),e_{\mathrm{S}_2})$.
  But clearly $((h,i),\delta)$ is not a square, and so we have a contradiction.
\end{proof}

\subsection{Extension}
Let $\tau \colon X^{(n)} \to X^{(n)}$ be an automorphism and assume that there exists a commutative diagram
\[
\begin{tikzcd}
X^n \setminus D \ar{d}{\rho} \ar{r}{f} & X^{n} \setminus D \ar{d}{\rho} \\
X^{(n)} \setminus \Delta \ar{r}{\tau} &  X^{(n)} \setminus \Delta.
\end{tikzcd}
\]
for some automorphism $f \in \Aut( X^n \setminus D )$. % such that the following diagram commutes:
\begin{prop} \label{prop_extension}
In the situation above the automorphism $f$ extends to an automorphism $\tilde{f} \colon X^n \to X^n$ such that $\tau \circ \rho = \rho \circ \tilde{f}$.
\end{prop}
\begin{proof}
Let $U = X^n \setminus D$ and $V = X^{(n)} \setminus \Delta$ and consider the diagram
\[
\begin{tikzcd}
U \ar{d}{\rho_U} \ar{r}{\iota} & X^n\ar{d}{\rho} \\
V \ar{r}{j} &  X^{(n)}
\end{tikzcd}
\]
where the horizontal maps $\iota$ and $j$ are the inclusions.
Because $U$ is of codimension $2$ and $X^n$ normal we have
$\iota_{\ast} \CO_U = \CO_{X^n}$.
%, hence
%\begin{equation} j_{\ast} \pi_{U \ast} \CO_U = \pi_{\ast} \iota_{\ast} \CO_U = \pi_{\ast} \CO_{X^n}. \label{extend123} \end{equation}
Both $\rho$ and $\rho_U$ are finite, therefore affine, so in particular we have
\[ X^n = \underline{\Spec}\, \rho_{\ast} \CO_{X^n}, \quad U = \underline{\Spec}\, \rho_{U \ast} \CO_{U}. \]

The category of affine schemes over a base $S$ is equivalent to the opposite of the category of quasi-coherent $\CO_S$-algebras.
Hence the $V$-morphism
\[
\begin{tikzcd} U \ar{rr}{f} \ar[swap]{dr}{\rho_U} & & U \ar{dl}{\tau^{-1} \circ \rho_U}\\
& V &
\end{tikzcd}
\]
corresponds to an isomorphism of quasi-coherent $\CO_V$-algebras
\[ \psi_f \colon (\tau^{-1})_{\ast} \rho_{\ast} \CO_U \to \rho_{\ast} \CO_U \]
By pushforward along $j$ we obtain an isomorphism of $\CO_{X^{(n)}}$-algebras
\[ j_{\ast} \psi_f \colon j_{\ast} \tau^{-1}_{\ast} \rho_{\ast} \CO_U \to j_{\ast} \rho_{\ast} \CO_U. \]
We have $j_{\ast} \rho_{\ast} \CO_U = \rho_{\ast} \iota_{\ast} \CO_U = \rho_{\ast} \CO_{X^n}$.
Moreover, because the automorphism $\tau$ of $U$ is the restriction of the automorphism $\tau$ of $X^n$ we also have
\[ j_{\ast} (\tau^{-1})_{\ast} \rho_{\ast} \CO_U = (\tau^{-1})_{\ast} j_{\ast} \rho_{\ast} \CO_U = (\tau_{-1})_{\ast} \rho_{\ast} \CO_{X^n}. \]
The pushforward $j_{\ast} \psi_{f}$ thus corresponds to an isomorphism $\tilde{f}$ from $\rho \colon X^n \to X^{(n)}$ to $\tau_{-1} \circ \rho \colon X^n \to X^{(n)}$,
hence to an isomorphism $\tilde{f} \in \Aut(X^n)$ with the desired properties.
%Its restriction to $U$ is $f$ definition, so it has the desired automorphism.
\end{proof}

\subsection{Splitting the automorphism}
\begin{prop} \label{Prop_Split_Auto}
Assume the surface $X$ is weak Fano or of general type,
and let $f \in \Aut(X^n)$ be an automorphism. Then at least one of the following holds:
\begin{enumerate}
\item[(a)] $f = \alpha \circ (f_1 \times \cdots \times f_n)$ for some $f_i \in \Aut(X)$ and $\alpha \in \mathrm{S}_n$, or
\item[(b)] $X \cong C_1 \times C_2$ for smooth curves $C_1, C_2$.
Moreover, if $C_1 \cong C_2$ then
\[ f = \alpha \circ (g_1 \times \cdots \times g_{2n}) \]
for some $g_i \in \Aut(C_1)$ and some $\alpha \in \mathrm{S}_{2n}$.
If $C_1 \ncong C_2$ then under the isomorphism $X^n \cong C_1^n \times C_2^n$ we have
\[ f \cong (\alpha_1 \times \alpha_2) \circ (g_1 \times \cdots \times g_n \times h_1 \times \cdots \times h_n) \]
for some $g_i \in \Aut(C_1), h_i \in \Aut(C_2)$ and $\alpha_1, \alpha_2 \in \mathrm{S}_n$.
%\item[(b)] $n=2$ and $X = C_1 \times C_2$ for smooth curves $C_1, C_2$, and up to composing with an automorphism of type (a), the map
%$f$ acts on $X^2 \cong C_1^2 \times C_2^2$ by $\mathrm{id}_{C_1^2} \times \alpha$ where $\alpha$ switches the factors on $C_2^2$.
\end{enumerate}
\end{prop}

For the proof we recall the general fact that the category of coherent sheaves on a smooth proper variety over an algebraically closed field is Krull--Schmidt,
that is every object in it can be uniquely decomposed in irreducible components \cite[Theorem 3]{At}.

Moreover, for a vector bundle $\mathcal{E}$ on $X$ we will write
\[ \mathcal{E}_i \coloneqq p_i^{\ast} \mathcal{E} \]
for the pullback of $\CE$ to $X^n$ along the projection to the $i$-th factor.
\begin{proof}
Let $X$ be a surface of general type.
We assume first that the cotangent bundle $\Omega_X$ is indecomposable.
In this case the bundle $\Omega_{X^n}$ has the Krull--Schmidt decomposition
\[ \Omega_{X^n} = \Omega_{X,1} \oplus \ldots \oplus \Omega_{X,n}. \]
Since this decomposition is unique and $\Omega_{X^n}$ is preserved under pullback by $f$,
for every $i$ we find
$f^{\ast} \Omega_{X,i} = \Omega_{X,j(i)}$ for some $j(i)$.
After composing $f$ with a
permutation we may assume
$f^{\ast} \Omega_{X,i} \cong \Omega_{X,i}$
and thus
\begin{equation} f^{\ast} \omega_{X,i} \cong \omega_{X,i}. \label{4gsdff} \end{equation}
Let $\varphi\colon X \to Y$ be the map to the canonical model of $X$ induced by a power of $\omega_{X}$.
From \eqref{4gsdff} we get an element $g_i \in \Aut(Y)$ such that the diagram
\[
\begin{tikzcd}
X^n \ar{d}{p_i} \ar{r}{f} & X^{n} \ar{d}{p_i} \\
%X \ar[dotted]{r}{\tilde{g_i}} \ar{d} &  X \ar{d} \\
Y \ar{r}{g_i} &  Y
\end{tikzcd}
\]
commutes; here $p_i$ is the composition of the projection to the $i$-th factor with $\varphi$.
Let $U \subset X$ be the open subset where $\varphi$ is an isomorphism.
Since $i$ was arbitrary we conclude
\[ (p_1, \ldots, p_n) \circ f|_{U^n} = g_1 \times \cdots \times g_n \]
For any fixed $(x_2, \ldots, x_n) \in U^{n-1}$ the composition
\[ f_1 \colon X
\hookrightarrow
%\lhook\joinrel\xrightarrow{x \mapsto (x, x_2, \ldots, x_n)}
%\hookxrightarrow{x \mapsto (x, x_2, \ldots, x_n)}
X^n \xrightarrow{f} X^n \xrightarrow{p_1} X, \]
where the first map is $x \mapsto (x,x_2, \ldots, x_n)$,
defines a lift of $g_1 \in \Aut(Y)$ to $f_1 \in \Aut(X)$.
Since the lift is unique if it exists, it is independent of the choice of the $x_i$.
Using a parallel argument we find lifts $f_i \in \Aut(X)$ of $g_i$ for any $i$. %be a lift of $g_i$.
%By a parallel argument for any $i$ we find a lift $f_i \in \Aut(X)$ of $g_i$.
The equality
\[ f = f_1 \times \cdots \times f_n \]
then holds on an non-empty open subset of $X^n$ %(namely $U \cap f(U)$)
and hence holds everywhere.

Assume now that the cotangent bundle of $X$ decomposes into line bundles:
\[ \Omega_X \cong \mathcal{L} \oplus \mathcal{M}. \]
By a result of Beauville \cite[\S 5.1, Proposition~4.3]{Bea3} the canonical bundle $\omega_X$ is ample and hence $W = X^n$ is canonically polarized.
By \cite[Theorem~1.3 and \S4]{BHPS}
it follows that the variety $W$ can be decomposed
into a product of irreducible factors and the decomposition is \emph{unique} up to reordering the factors.
Here a variety $Z$ is called irreducible it does not admit a non-trivial product decomposition $Z \cong Z_1 \times Z_2$.

If $X$ is not the product of curves then $W$ has the following two factorizations.
The standard one, induced by the projection maps $p_i$,
%The first is the de by the product of the projection maps $p_i : W \to X$,
\[ p = (p_1, \ldots,p_n) \colon W \xrightarrow{\sim} X^n \]
and the one obtained by mapping the first under the automorphism $f$,
\[ p \circ f \colon W \xrightarrow{\sim} X^n. \]
Since both must coincide up to reordering (and factorwise isomorphism)
we conclude there exist automorphisms $f_i \in \Aut(X)$ and a permutation $\alpha \in \mathrm{S}_n$ such that
$p_i \circ f = f_i \circ p_{\alpha(i)}$ for all $i$. % for some automorphism $f_i \in \Aut(X)$.
This yields the claim.

We therefore can assume that $X$ is the product of curves
\[ X = C_1 \times C_2 \]
and hence that
\[ \mathcal{L} = q_1^{\ast} \Omega_{C_1}, \quad \mathcal{M} = q_2^{\ast} \Omega_{C_2} \]
where we let $q_j$ denote the projection from $X$ to the $j$-th factor.
%\[ L_{1} = \mathrm{proj}_1^{\ast} \Omega_{C_1}, \quad L_{2} = \mathrm{proj}_2^{\ast} \Omega_{C_2} \]
%where we let $\mathrm{proj}_j$ denote the projection from $X$ to the $j$-th factor.

If $C_1 \ncong C_2$ then the pullback by $f$ preserves the set of $\mathcal{L}_{i}$ and the set of $\mathcal{M}_{i}$ separately
%If $C_1 \ncong C_2$ then the pullback by $f$ preserves the set of $L_{1,i}$ and the set of $L_{2,i}$ separately
(since the image of the complete linear system defined by a power of $\CL$ is precisely $C_1$, and likewise for $\CM$).
Hence there exists $g_i \in \Aut(C_1)$ and a permutation $\alpha_1 \in \mathrm{S}_n$ such that
the following diagram commutes:
\[
\begin{tikzcd}
X^n \ar[swap]{d}{q_1 \circ p_i} \ar{r}{f} & X^{n} \ar{d}{q_1 \circ p_{\alpha_1(i)}} \\
C_1 \ar{r}{g_i} &  C_1.
%X^n \ar[swap]{d}{\mathrm{proj}_1 \circ p_i} \ar{r}{f} & X^{n} \ar{d}{\mathrm{proj}_1 \circ p_{\alpha_1(i)}} \\
%C_1 \ar{r}{g_i} &  C_1.
\end{tikzcd}
\]
Since the parallel statement holds for the factor $C_2$, this yields the claim.

If $C_1 \cong C_2$ then we may determine $\Aut(C_1^{2n})$ as we determined $\Aut(X^n)$ when $\Omega_X$ is indecomposable,
or again apply the result \cite[Theorem~1.3]{BHPS}.

Finally, we consider the case where $X$ is weak Fano.
If $\Omega_X$ is indecomposable, then we can argue as for general type.
If $\Omega_X$ decomposes, then by the classification of weak Fano surfaces, or using Beauville \cite{Bea3} and that $X$ is rational so simply-connected,
we have that $X$ is isomorphic to $\p^1 \times \p^1$. The claim then follows as in case $C_1 = C_2$ above.
%
%This gives the decomposition
%\[ \Omega_{X^n} = \bigoplus_{i=1}^{n} L_{1,i} \oplus L_{2,i}. \]
%%Since pullback by $f$ preserves the cotangent bundle,
%Invoking Krull--Schmidt we find that $f^{\ast}$ acts on the set of line bundles $L_{1,i}, L_{2,i}$, $i=1,\ldots, n$.
%%
%%$f^{\ast}
%%\[ f^{\ast} L_{1,i} = L_{1,i} \text{ or } L_{2,i} \]
%%for all $i$
\end{proof}

\begin{cor} \label{Cor_split}
Assume the surface $X$ is weak Fano or of general type. Let $f \in \Aut(X^n)$ and $\tau \in \Aut(X^{(n)})$
be automorphisms such that the diagram
\[
\begin{tikzcd}
X^n \ar{d}{\rho} \ar{r}{f} & X^{n} \ar{d}{\rho} \\
X^{(n)} \ar{r}{\tau} &  X^{(n)}
\end{tikzcd}
\]
commutes. Then one of the following holds:
\begin{enumerate}
\item[(a)]
$f = \alpha \circ (g \times \cdots \times g)$ for some $g \in \Aut(X)$ and $\alpha \in \mathrm{S}_n$, or
\item[(b)] $n=2$ and $X = C_1 \times C_2$ for smooth curves $C_1, C_2$, and up to composing with an automorphism as in (a), the map
$f$ acts on $X^2 \cong C_1^2 \times C_2^2$ by $\mathrm{id}_{C_1^2} \times \alpha$ where $\alpha \in \mathrm{S}_2$ switches the factors on $C_2^2$.
\end{enumerate}
\end{cor}
\begin{proof}
Since the automorphism $\tau$ preserves the small diagonal in $X^{(n)}$
the map $f$ preserves the small diagonal
\[ D_{\text{small}} = \{ (x,\ldots, x) \mid x \in X \} \subset X^n. \]

If $f$ is as in Proposition~\ref{Prop_Split_Auto} (a) this
immediately implies the claim.
If $f$ is as in part (b) we consider the two cases more closely.
In case
$C_1 = C_2$ then for all $(c_1, c_2) \in X$ there exist $(d_1, d_2)$ such that
\[ \alpha \circ (g_1(c_1), g_2(c_2), \ldots, g_{2n-1}(c_1), g_{2n}(c_2) ) = (d_1, d_2, \ldots, d_1, d_2). \]
By choosing $c_1$ in the open subset of $C_1$ such that $g_1(c_1)$ is distinct from all the $g_2(c_2), g_4(c_2), \ldots, g_{2n}(c_2)$
for some fixed $c_2$ we find that
\[ g_1 = g_3 = \ldots = g_{2n-1}. \]
Similarly we have $g_2 = g_4 = \ldots = g_{2n}$.
Moreover, $\alpha$ must preserve parity.
In the case $C_1 \neq C_2$ we similarly find $g_1 = \ldots = g_n$ and $h_1 = \ldots = h_n$.

Under the identification $X^n = C_1^n \times C_2^n$ we conclude that
\[ f = (\alpha_1 \times \alpha_2) (g^{\times n} \times h^{\times n}) \]
for some $\alpha_1, \alpha_2 \in \mathrm{S}_n$ and $g \in \Aut(C_1), h \in \Aut(C_2)$.
After composing with an automorphism as in part (a) we hence can assume
\begin{equation} f = (\mathrm{id}_{C_1^n}, \alpha_2). \label{fdfs} \end{equation}

Since the automorphism $f$ descends to the symmetric product, for a point $z \in X^n$ the class of $f(z)$ in $X^{(n)}$
depends only on the class of $z \in X^{(n)}$.
Hence for all permutations $\sigma \in \mathrm{S}_n$ there exists $\tilde{\sigma} \in \mathrm{S}_n$ such that
\[ f(\sigma z) = \tilde{\sigma} f(z). \]
Using \eqref{fdfs} and remembering that under the identification $X^n \cong C_1^n \times C_2^n$ the symmetric group $\mathrm{S}_n$ acts diagonally we find
% P: need identity on C_1^n rather, right?
\[ (\tilde{\sigma} \times \tilde{\sigma})^{-1} (\mathrm{id}_{C_1^n} \times \alpha_2) (\sigma \times \sigma) = (\mathrm{id}_{C_1^n}, \alpha_2). \]
%\[ (\tilde{\sigma} \times \tilde{\sigma})^{-1} (\mathrm{id}_{\mathrm{S}_n} \times \alpha_2) (\sigma \times \sigma) = (\mathrm{id}_{\mathrm{S}_n}, \alpha_2). \]
This gives $\sigma = \tilde{\sigma}$ and thus
\[ \sigma^{-1} \alpha \sigma = \alpha \]
Since $\sigma$ was arbitrary, $\alpha$ is in the center of $\mathrm{S}_n$. Since the center of the symmetric group is non-trivial only for $n=2$ the claim follows. % so $n=2$ and the claim follows.
\end{proof}

\subsection{Proof of Theorem~\ref{thm1} and \ref{thm2}}
Let $X$ be weak Fano or of general type and let $\sigma \in \Aut(X^{[n]})$.
%If the surface $X$ is Fano then the automorphism descends to the symmetric product by the discussion in Section~\ref{section_p2}.
By Proposition~\ref{Prop_Sym2} applied to Proposition~\ref{Prop_Sym1} the automorphism $\sigma$ descends to the symmetric product.
By Proposition~\ref{prop_lifting} this automorphism of the symmetric product lifts to an automorphism of the complement of the big diagonal in $X^n$
and by Proposition~\ref{prop_extension} it extends from there to an automorphism of $X^n$.
Theorem~\ref{thm1} and the uniqueness part of Theorem~\ref{thm2} now follow from the classification in Corollary~\ref{Cor_split}.
For the existence part of Theorem~\ref{thm2} the automorphism in Corollary~\ref{Cor_split} (b)
descends to $X^{(2)}$ and from there lifts to the Hilbert scheme by the universal property of the blow-up $X^{[2]} \to X^{(2)}$.
\qed

\subsection{Proof of Corollary~\ref{Cor1}}
Theorem~\ref{thm1} together with the following result immediately implies Corollary~\ref{Cor1}.

\begin{prop}
Let $X,Y$ be smooth projective surfaces and let $Y$ be weak Fano or of general type.
If $X^{[n]} \cong Y^{[n]}$, then $X \cong Y$.
\end{prop}
%
%Let $X,Y$ be smooth projective surfaces and let
%\[ \sigma: X^{[n]} \overset{\sim}{\to} Y^{[n]} \]
%be an isomorphism.
%We assume that $Y$ is of general type, and that if $Y$ is isomorphic to a product of curves, then $n \geq 3$.
%The case where $Y$ is Fano is parallel and omitted.
%
%We will show that $X \cong Y$ from which the claim follows by Theorem~\ref{thm1}.

\begin{proof}
We assume that $Y$ is of general type, the case where $Y$ is weak Fano is parallel.
Let $\sigma \colon X^{[n]} \overset{\sim}{\to} Y^{[n]}$ be an isomorphism.
Since $\sigma^{\ast} \omega_{Y^{[n]}} \cong \omega_{X^{[n]}}$ we have a commutative diagram
\[
\begin{tikzcd}
X^{[n]} \ar{d} \ar{r}{\sigma} & Y^{[n]} \ar{d} \\
X_{\text{can}}^{(n)} \ar{r}{\tau} &  Y_{\text{can}}^{(n)}
\end{tikzcd}
\]
where $\tau$ is an isomorphism and we let $X_{\text{can}}$ and $Y_{\text{can}}$ denote the canonical models of $X$ and $Y$ respectively.
Since $Y_{\text{can}}^{(n)}$ is of dimension $2n$ we find that $X$ is also of general type.
Let $x_i$ and $y_j$ be the singular points of $X_{\text{can}}$ and $Y_{\text{can}}$ respectively.
Then $\tau$ induces an isomorphism of the singular loci of $X_{\text{can}}^{(n)}$ and $Y_{\text{can}}^{(n)}$:
\[ \tau \colon \bigcup_i D_{x_i} \cup \Delta_{X_\text{can}} \, \xrightarrow{\sim} \, \bigcup_j D_{y_j} \cup \Delta_{Y_\text{can}}. \]
where the $D_x$ are defined as in \eqref{3452354}.
We claim that $\tau(\Delta_{X_{\text{can}}}) = \Delta_{Y_{\text{can}}}$.
Indeed, if $n=2$ then $D_{x_i} \cong \Delta_{X_{\text{can}}} \cong X_{\text{can}}$, so $X_{\text{can}} \cong Y_{\text{can}}$
and the claim follows as in the proof of Proposition~\ref{Prop_Sym2}.
If $n \geq 3$, we can either use that $D_{x_i}$ is normal while the diagonal is not, or
argue as follows. Assume that
\[ \tau( \Delta_{X_{\text{can}}}) = D_{y_i}, \quad \tau( \Delta_{Y_{\text{can}}} ) = D_{x_i}. \]
Then from the description of their normalizations in \eqref{kggh} we have the equality of Betti numbers %Proposition~\ref{Prop_Sym2}
\[ 2 \mathrm{b}_i(X_{\text{can}}) = \mathrm{b}_i(Y_{\text{can}}) \ \ \text{ and } \ \ \mathrm{b}_i(X_{\text{can}}) = 2 \mathrm{b}_i(Y_{\text{can}}) \]
for $i=1$ so $\mathrm{b}_1(X_{\text{can}}) = \mathrm{b}_1( Y_{\text{can}}) = 0$, hence the same equations hold also for $i=2$ which is impossible since $\mathrm{b}_2(X_{\text{can}}) > 0$.
%where we have
%written $\mathrm{h}^i(Z)$ for the dimension of the cohomology group $\mathrm{H}^i(Z,\BQ)$.

By the claim, $\tau$ preserves the diagonal, hence by an argument parallel to the proof of Proposition~\ref{Prop_Sym2}, we find that
$\sigma$ preserves the class of a $\p^1$-fiber of the Hilbert--Chow morphism $X^{[n]} \to X^{(n)}$.
Then arguing as in Proposition~\ref{Prop_Sym1} we find that $\sigma$ descends to an isomorphism
$X^{(n)} \overset{\sim}{\to} Y^{(n)}$.
Since this automorphism sends the small diagonal to the small diagonal and the small diagonal is isomorphic to the underlying surface we are done.
\end{proof}

\section{The Hilbert scheme of \texorpdfstring{$2$}{2} points of \texorpdfstring{$\p^n$}{projective space}}
The Hilbert scheme of $2$ points of $\p^n$ is isomorphic to the quotient of the blow-up of $\p^n \times \p^n$ along the diagonal,
\[ (\p^n)^{[2]} = \mathrm{Bl}_{\Delta}( \p^n \times \p^n ) / \mathrm{S}_2. \]
%The involution by which the quotient is taken is the lift of the map that interchanges factors.
Since $(\p^n)^{[2]}$ is rational we find
\[ \Pic (\p^n)^{[2]} = \mathrm{H}^2( (\p^n)^{[2]} , \BZ) \cong \BZ^{\oplus2}. \]

\begin{proof}[Proof of Theorem~\ref{Hilb2Pn}]
Let $\sigma \colon (\p^n)^{[2]} \overset{\sim}{\to} (\p^n)^{[2]}$ be an automorphism.
The Hilbert scheme admits the following two contractions: the Hilbert--Chow morphism
\[ f_1 \colon (\p^n)^{[2]} \to (\p^n)^{(2)}  \]
and the morphism
\[ f_2 \colon (\p^n)^{[2]} \to \mathrm{Gr}(2,n+1) \]
that sends a subscheme to the line passing through it.
By pulling back polarizations of the targets along $f_1, f_2$ we obtain two divisors on the Hilbert scheme.
Both maps contract curves so both divisors lie in the boundary of the nef cone of $(\p^n)^{[2]}$.
Since the Picard group is rank $2$ these divisors form precisely the extremal rays of the nef cone.

Since the automorphism $\sigma$ preserves the nef cone, $\sigma$ up to scaling either preserves these divisors or interchanges them.
However, since the contractions above are non-isomorphic (they have non-isomorphic images), $\sigma$ cannot interchange them hence must preserve them up to scaling.
Since $\sigma$ also preserves the divisibility, we find that $\sigma$ fixes the two divisors, so $\sigma$ fixed the Picard group.
In particular, from
\[ \sigma^{\ast} f_1^{\ast} L = f_1^{\ast} L \]
where $L$ is an ample divisor on $(\p^n)^{(2)}$ we find that $\sigma$ descends to
an automorphism of the symmetric product $\tau \colon (\p^n)^{(2)} \to (\p^n)^{(2)}$.
Arguing as in Section~\ref{section_p2} this automorphism lifts to $\p^n \times \p^n$ where it has to be of the form $\alpha \circ (f,f)$ for some $\alpha \in \mathrm{S}_2$ and $f \in \Aut( \p^n )$.
Hence $\sigma$ is natural.
\end{proof}

\end{document}